\documentclass[a4paper]{article} 
\usepackage{url} 

\usepackage{amsmath,amsfonts,amssymb,amsthm,graphicx,enumerate,mathrsfs,color,subfigure,multicol,mathdots,booktabs,xypic}
\usepackage[noend]{algpseudocode}
\usepackage{algorithmicx,algorithm}
\newtheorem{definition}{Definition}

\newtheorem{theorem}{Theorem}
\newtheorem{corollary}{Corollary}
\newtheorem{lemma}{Lemma}

\begin{document}\large
\title{M-eigenvalues of Riemann Curvature Tensor of Conformally Flat Manifolds}
\author{Yun Miao\footnote{E-mail: 15110180014@fudan.edu.cn. School of Mathematical Sciences, Fudan University, Shanghai, 200433, P. R. of China. Y. Miao is supported by the National Natural Science Foundation of China under grant 11771099. } \quad  Liqun Qi\footnote{ E-mail: maqilq@polyu.edu.hk. Department of Applied Mathematics, the Hong Kong Polytechnic University, Hong Kong. L. Qi is supported by the Hong Kong Research Grant Council (Grant No. PolyU 15302114, 15300715, 15301716 and 15300717)} \quad  Yimin Wei\footnote{Corresponding author. E-mail: ymwei@fudan.edu.cn and yimin.wei@gmail.com. School of Mathematical Sciences and Key Laboratory of Mathematics for Nonlinear Sciences, Fudan University, Shanghai, 200433, P. R. of China. Y. Wei is supported by the National Natural Science Foundation of China under grant 11771099 and International Cooperation Project of Shanghai Municipal Science and Technology Commission under grant 16510711200.
}}
\maketitle

\begin{abstract}
We generalized Xiang, Qi and Wei's results on the M-eigenvalues of Riemann curvature tensor to higher dimensional conformal flat manifolds. The expression of M-eigenvalues and M-eigenvectors are found in our paper. As a special case, M-eigenvalues of conformal flat Einstein manifold have also been discussed, and the conformal the invariance of M-eigentriple has been found. We also discussed the relationship between M-eigenvalue and sectional curvature of a Riemannian manifold. We proved that the M-eigenvalue can determine the Riemann curvature tensor uniquely and generalize the real M-eigenvalue to complex cases. In the last part of our paper, we give an example to compute the M-eigentriple of de Sitter spacetime which is well-known in general relativity.
\bigskip

  \hspace{-14pt}{\bf Keywords.} M-eigenvalue, Riemann Curvature tensor, Ricci tensor, Conformal invariant, Canonical form.

  \bigskip

  \hspace{-14pt}{\bf AMS subject classifications.} 15A48, 15A69, 65F10, 65H10, 65N22.

\end{abstract}

\section{Introduction}
The eigenvalue problem is a very important topic in tensor computation theories. Xiang, Qi and Wei \cite{Xiang, Xiang2} considered the M-eigenvalue problem of elasticity tensor and discussed the relation between strong ellipticity condition of a tensor and it's M-eigenvalues, and then extended the M-eigenvalue problem to Riemann curvature tensor.\par
Qi, Dai and Han \cite{Qi2} introduced the M-eigenvalues for the elasticity tensor. The M-eigenvalues $\theta$ of the fourth-order tensor $E_{ijkl}$ are defined as follows.
$$
E_{ijkl}y^j x^k y^l =\theta x_i,
$$
$$
E_{ijkl}x^i y^j x^k =\theta y_l,
$$
under constraints $x^{\top}x=1$ and $y^{\top}y=1$. Such kind of eigenvalue is closely related to the strong ellipticity and positive definiteness of the material respectively.\par
Recently, Xiang, Qi and Wei \cite{Xiang} introduced the M-eigenvalue problem for the Riemann curvature tensor as follows,
$$
R^{i}_{\;jkl} y^j x^k y^l=\lambda x^i.
$$
where $x^{\top}x=1$ and $y^{\top}y=1$.
They also calculated several simple cases such as two dimensional and three dimensional case to examine what the M-eigenvalues are. They have found that in low dimension cases, the M-eigenvalues are closely related to the Gauss curvature and scalar curvature according to the the expression of Riemann curvature tensor, Ricci curvature tensor and Theorema Egregium of Gauss. Case of constant curvature has also been considered.\par
The intrinsic structures of manifolds are important to differential geometry, and in the last century, many results and theorems have been developed by S.S. Chern, we further discussed the M-eigenvalue of Riemann curvature tensor according to the intrinsic geometry structure, and generalize Xiang, Qi and Wei's results to conformal flat manifolds in our paper. \par
This paper is organized as follows. In the preliminaries, we listed the formulas and results in Riemannian geometry, and concluded the symmetries of Riemann curvature tensor and give the definition of the M-eigenvalue problem of Riemann curvature tensor. In the main part of our paper, We first review the result get by Xiang, Qi and Wei \cite{Xiang} in two dimension manifold and three dimension manifold cases, then we introduced the Ricci decomposition of Riemann curvature tensor and give the definition of a conformal flat manifold. For conformal flat manifold of dimension $m$, we find the Riemann curvature tensor has explicit expression which is related to the scalar curvature and Ricci curvature tensor. Using this expression we give the M-eigenvalue of higher dimensional cases of Riemann curvature tensor. In the next subsection, we further discussed the M-eigenvalue problem for conformal flat Einstein manifold. We found that the M-eigenvalue and M-eigenvector is conformal invariant up to the conformal factor.\par
Since sectional curvature is of great importance to Riemannian manifold, in the next subsection of our main part, we discussed the relation between M-eigenvalue and sectional curvature and found that the M-eigenvalues can determine the Riemann curvature tensor uniquely. According to this theorem, the canonical form of a Riemann curvature tensor can be introduced. As we know, the H-eigenvalue and Z-eigenvalue may not have this kind of characteristics to determine a tensor. Then we give the expression of the M-eigentriple of Riemann curvature tensor of any conformal flat $m$ dimensional manifold. Since the complex geometry is of great importance in modern mathematics, we generalize the M-eigenvalue of Riemann curvature tensor of a differential manifold to complex curvature tensor of K$\ddot{\rm{a}}$hler manifolds, and we also give an example of the complex M-eigenvalue of a constant curvature manifold.\par
In the final part of our paper, we give the example to compute the M-eigenvalue and M-eigenvector of Riemann curvature tensor of de Sitter Spacetime in the general relativity, which is well known to relativistics, also we found that the M-eigenvalues and M-eigenvectors can determine the Riemann curvature tensor uniquely.
\section{Preliminaries}
\subsection{List of formulas in Riemannian geometry}
Suppose $(M,g)$ is a Riemannian manifold, $g$ is the metric tensor of $M$. $\nabla$ is the Riemannian connection, the curvature operator and the curvature tensor $R$ as follows:
$$
R:T_p M \times T_p M \times T_p M \times T_p M \rightarrow \mathbb{R}
$$
$$
R(X,Y)Z=\nabla_X\nabla_YZ-\nabla_Y\nabla_XZ-\nabla_{[X,Y]}Z,
$$
$$
R(X,Y,Z,W)=\langle R(Z,W)Y,X \rangle,
$$
where $X,Y,Z,W\in \mathscr{X}(M)$.\par
\begin{lemma}
For any vector field $X,Y,Z,W\in \mathscr{X}(M)$, the curvature operator and curvature tensor satisfies the following relations:
$$
\begin{cases}
\rm{(i)}&R(X,Y)Z+R(Y,X)Z=0,\\
\rm{(ii)}&R(X,Y)Z+R(Y,Z)X+R(Z,X)Y=0,\\
\rm{(iii)}&R(X,Y,Z,W)=-R(Y,X,Z,W)=R(X,Y,W,Z),\\
\rm{(iv)}&R(X,Y,Z,W)=R(Z,W,X,Y).\\
\end{cases}
$$
\end{lemma}
Suppose $\{e_1,e_2,\cdots,e_m\}$ is the normalized linear independent local coordinate chart, we have
$$
g_{ij}=g(e_i,e_j),
$$
which is the component of metric tensor, and
$$
R_{ijkl}=R(e_i,e_j,e_k,e_l)=g_{ih}R^{h}_{\; jkl},
$$
$$
R^{i}_{\; jkl}e_i=R(e_k,e_l)e_j.
$$
We usually call the second equation the first Bianchi Identity.\par
We have the component form of the properties of curvature tensor:
$$
\begin{cases}
\rm{(i)}&R^{i}_{\;jkl}+R^{i}_{\;jlk}=0,\\
\rm{(ii)}&R^{i}_{\;jkl}+R^{i}_{\;klj}+R^{i}_{\;ljk}=0,\\
\rm{(iii)}&R_{ijkl}=-R_{jikl}=-R_{ijlk},\\
\rm{(iv)}&R_{ijkl}=R_{klij}.\\
\end{cases}
$$
The second identity is the first Bianchi idendity componentwise.\par
The contraction yields the Ricci tensor $R_{ik}$ and the scalar tensor $R$ as follows.
$$
R_{ik}=g^{hj}R_{hijk}=R^{m}_{\;imk},
$$
$$
R=g^{ik}R_{ik}=R^{k}_{\;k}=g^{ik}R^{m}_{\;imk}=g^{ik}g^{hj}R_{hijk},
$$
where $g^{ij}$ is the inverse matrix of metrix tensor $g_{ij}$.\par

\section{M-eigenvalue problem of Riemann curvature tensor}
\subsection{Results of Xiang, Qi and Wei}

Xiang, Qi and Wei \cite{Xiang} studied the several simple cases on the M-eigenvalue problem and examined what the M-eigenvalues will be. In this section, we will generalize his results to higher dimensions.\par
We can give the definition of M-eigenvalues and M-eigenvector as follows:
\begin{definition}
The M-eigenvalue and M-eigenvector of Riemann curvature tensor of a manifold $(M,g)$ is defined as follows:
$$
R^{i}_{\;jkl} y^j x^k y^l=\lambda x^i.
$$
where $x^{\top}x=1$ and $y^{\top}y=1$.
\end{definition}
For 2D case, the Riemann curvature tensor can be expressed by
$$
R_{abcd}=K(g_{ac}g_{bd}-g_{ad}g_{bc})
$$
where K is the Gaussian curvature.
\begin{corollary}
For the 2D case, the modified M-eigenvalues are $\zeta=K$ and $K/2$，where $K$ is the Gaussian curvature.
\end{corollary}
This result can be get easily from the M-eigenvalue problem equation.\par
For 3D case, Xiang, Qi and Wei \cite{Xiang} use the expression of the curvature tensor
$$
R_{abcd}=R_{ac}g_{bd}-R_{ad}g_{bc}+g_{ac}R_{bd}-g_{ad}R_{bc}-\frac{R}{2}(g_{ac}g_{bd}-g_{ad}g_{bc})
$$
where $R_{ab}$ is Ricci curvature tensor.
\begin{corollary}
For the 3D case, if $\lambda$ and $\mu$ are the eigenvalues of Ricci tensor associated with eigenvectors $x$ and $y$ respectively, then $(\zeta,x,y)$ is the eigentriple of the M-eigenvalue problem with M-eigenvalue
$$
\zeta=\lambda+\mu+\frac{R}{2}.
$$
\end{corollary}
In order to get more further results and characteristics on the eigenvalue problem, we need to consider more intrinsic structure on the manifold, which is enlightened by S.S. Chern and his famous research
 in differential geometry.
\subsection{Ricci decomposition}
In semi-Riemannian geometry, the Ricci decomposition is a way of breaking up the Riemann curvature tensor of a pseudo-Riemannian manifold into pieces with useful individual algebraic properties. This decomposition is of fundamental importance in Riemannian- and pseudo-Riemannian geometry.\par
\begin{lemma}
The Ricci decomposition of Riemann curvature tensor for a n dimensional semi-Riemannian manifold is
$$
R_{abcd}=S_{abcd}+E_{abcd}+C_{abcd}
$$
The three pieces are:\newline
1. The scalar part
$$
S_{abcd}=\frac{R}{n(n-1)}H_{abcd}
$$
is built by the scalar curvature $R=R_{\;m}^{m}$, where $R_{ab}=R^{c}_{\;acb}$ is the Ricci curvature tensor, and
$$
H_{abcd}=g_{ac}g_{db}-g_{ad}g_{cb}=2g_{a[c}g_{d]b}
$$
2. The semi-traceless part
$$
\begin{aligned}
E_{abcd}&=\frac{1}{n-2}(g_{ac}S_{bd}-g_{ad}S_{bc}+g_{bd}S_{ac}-g_{bc}S_{ad})\\
        &=\frac{2}{n-2}(g_{a[c}S_{d]b}-g_{b[c}S_{d]a})
\end{aligned}
$$
where
$$
S_{ab}=R_{ab}-\frac{1}{n}g_{ab}R
$$
3. The Weyl tensor $C_{abcd}$ or the conformal curvature tensor, which is complete traceless, in the sense that taking the trace, or contraction, over any pair of indices gives zero.
\end{lemma}
The Ricci decomposition can be interpreted physically in Einstein's theory of general relativity, where it is sometimes called the Géhéniau-Debever decomposition. In this theory, the Einstein field equation. In this theory, the Einstein field equation
$$
G_{ab}=8\pi T_{ab}
$$
where $T_{ab}$ is the stress$-$energy tensor describing the amount and motion of all matter and all nongravitational field energy and momentum, states that the Ricci tensor, or equivalently, the Einstein tensor represents that part of the gravitational field which is due to the immediate presence of nongravitational energy and momentum. The Weyl tensor represents the part of the gravitational field which can propagate as a gravitational wave through a region containing no matter or nongravitational fields. Regions of spacetime in which the Weyl tensor vanishes contain no gravitational radiation and are also conformally flat.

\subsection{M-eigenvalue of Riemann curvature tensor on higher dimension}
\begin{definition}
If the conformal part of Riemann curvature tensor
$$
C_{abcd}\equiv 0
$$
we call a manifold is conformally equivalent to a flat manifold.
\end{definition}
\begin{theorem}
If a m dimensional manifold (M,g) is conformally equivalent to a flat manifold, then the Riemann curvature tensor will have the expression
$$
\begin{aligned}
R^{l}_{\;kij}&=\frac{1}{m-2}(\delta^{l}_{\;i}R_{kj}-\delta^{l}_{\;j}R_{ki}
+g_{kj}g^{lp}R_{pi}-g_{ki}g^{lp}R_{pj})\\
&+\frac{R}{(m-1)(m-2)}(\delta^{l}_{\;j}g_{ki}-\delta^{l}_{\;i}g_{kj})
\end{aligned}
$$
\end{theorem}
\begin{proof}
Since $(M,g)$ is conformally equivalent to a flat manifold, we have $$C_{abcd}\equiv 0.$$ Then the result can be get from the Ricci decomposition easily.
\end{proof}
Suppose that $(\lambda, x)$ and $(\mu, y)$ are the eigenpairs of Ricci curvature tensor $(x\neq y)$, that is $R_{ac}x^{c}=\lambda x_{a}$ and $R_{ad}y^{d}=\mu y_{a}$. The M-eigenvalues of Riemann curvature tensor on higher dimension Riemannian manifold can be stated as follows.
\begin{theorem}
Suppose $(M,g)$ is a m dimensional conformal flat Riemannian manifold, if $\lambda$ and $\mu$ are the eigenvalues of Ricci tensor associated with eigenvectors $x$ and $y$, respectively, which satisfies
$$
\langle x,x\rangle=g_{ab}x^{a}x_{b}=1,
$$
$$
\langle y,y\rangle=g_{ab}y^{a}y_{b}=1,
$$
$$
\langle x,y\rangle=g_{ab}x^{a}y_{b}=0,
$$
then $(\zeta, x, y)$ is the eigentriple of the modified M-eigenvalue problem with the M-eigenvalue
$$
\zeta=\frac{(m-1)(\lambda+\mu)+R}{(m-2)(m-1)}
$$
\end{theorem}
\begin{proof}
For the M-eigenvalue problem, we need to calculate
$$
R^{l}_{kij}y^{k}x^{i}y^{j}=\zeta x^{l}.
$$
The left hand side of the equation equals to
$$
\begin{aligned}
R^{l}_{\;kij}&=\frac{1}{m-2}(\delta^{l}_{\;i}R_{kj}-\delta^{l}_{\;j}R_{ki}
+g_{kj}g^{lp}R_{pi}-g_{ki}g^{lp}R_{pj})y^{k}x^{i}y^{j}\\
&+\frac{R}{(m-1)(m-2)}(\delta^{l}_{\;j}g_{ki}-\delta^{l}_{\;i}g_{kj})y^{k}x^{i}y^{j}.
\end{aligned}
$$
The first term equals to
$$
\begin{aligned}
R^{l}_{\;kij}y^{k}x^{i}y^{j}&=\frac{1}{m-2}(\delta^{l}_{\;i}R_{kj}-\delta^{l}_{\;j}R_{ki}
+g_{kj}g^{lp}R_{pi}-g_{ki}g^{lp}R_{pj})y^{k}x^{i}y^{j}\\
&=\frac{1}{m-2}(\delta^{l}_{\;i}R_{kj}y^{k}x^{i}y^{j}-\delta^{l}_{\;j}R_{ki}y^{k}x^{i}y^{j}
+g_{kj}g^{lp}R_{pi}y^{k}x^{i}y^{j}-g_{ki}g^{lp}R_{pj}y^{k}x^{i}y^{j})\\
&=\frac{1}{m-2}(\mu \langle y,y\rangle x^{l}-\mu \langle x,y\rangle y^{l}+g^{lp}\langle y,y\rangle \lambda x_{p}-\langle x,y\rangle g^{lp} \mu y_{p})\\
&=\frac{1}{m-2}(\mu x^{l}+\lambda x^{l})\\
&=\frac{\mu+\lambda}{m-2} x^{l}.
\end{aligned}
$$
the second term equals to
$$
\begin{aligned}
\frac{R}{(m-1)(m-2)}(\delta^{l}_{\;j}g_{ki}y^{k}x^{i}y^{j}-\delta^{l}_{\;i}g_{kj}y^{k}x^{i}y^{j})=\frac{R}{(m-1)(m-2)}x^{l}.
\end{aligned}
$$
So we have
$$
R^{l}_{\;kij}y^{k}x^{i}y^{j}=\frac{(m-1)(\lambda+\mu)+R}{(m-2)(m-1)}x^{l}.
$$
The second equation of M-eigenproblem reads
$$
R^{l}_{\;kij}x^{k}y^{i}x^{j}=\zeta y^{l}
$$
is similar to above.
Therefore we have $\zeta=\frac{(m-1)(\lambda+\mu)+R}{(m-2)(m-1)}$ is the M-eigenvalue of the Riemann curvature tensor.
\end{proof}
We find Xiang, Qi and Wei's result is only the special case of our theorem according to the following lemma.
\begin{lemma}
If a manifold is of dimension three, then it must be conformally equivalent to a flat manifold, which means the Weyl tensor
$$
C_{abcd}\equiv 0
$$
of a three dimensional manifold.
\end{lemma}
\begin{proof}
To see that $C_{ijkl}\equiv 0$ in dimension three, observe first that it has all the symmetries of the Riemmanian curvature tensor $R_{ijkl}$, so that
$$
C_{ijkl}=-C_{jikl}=-C_{ijlk}=C_{jilk}=C_{klij},
$$
$$
C_{ijkl}+C_{iklj}+C_{iljk}=0
$$
and in addition all its traces vanish, that is,
$$
g^{ik}C_{ijkl}=0.
$$
Thus
$$
C_{1111}+C_{1212}+C_{1313}=0,
$$
and so
$$
C_{1212}=-C_{1313}=C_{2323}=-C_{2121}=-C_{1212}
$$
which implies $C_{1212}=0$. Moreover
$$
C_{1213}+C_{2223}+C_{3233}=0
$$
and so $C_{1213}=0$ also. Hence in general any term $C_{ijkl}=0$ unless $i,j,k$ and $l$ are all distinct. In dimension 3 there are only 3 possible choices for the indices, and the tensor must vanish identically.
\end{proof}
\begin{corollary}
For a three dimensional Rimannian manifold (M,g),
$$
\zeta=\mu+\lambda+\frac{R}{2}
$$
is a M-eigenvalue of the Riemannian curvature tensor.
\end{corollary}
\begin{proof}
take $m=3$ into
$$
\zeta=\frac{(m-1)(\lambda+\mu)+R}{(m-2)(m-1)}
$$
then we get the result.
\end{proof}
\begin{lemma}
A constant curvature manifold is conformal flat, which means the Weyl tensor
$$
C_{abcd}\equiv 0.
$$
\end{lemma}

\subsection{Einstein manifold}
In differential geometry and mathematical physics, an Einstein manifold is a Riemannian or pseudo-Riemannian differentiable manifold whose Ricci tensor is proportional to the metric. They are named after Albert Einstein because this condition is equivalent to saying that the metric is a solution of the vacuum Einstein field equations (with cosmological constant), although both the dimension and the signature of the metric can be arbitrary, thus not being restricted to the four-dimensional Lorentzian manifolds usually studied in general relativity.
\begin{definition}
We call a m-dimensional manifold $(M,g)$ is an Einstein manifold, if and only if
$$
R_{ab}=k g_{ab}.
$$
for some constant $k$, where $R_{ab}$ is the Ricci curvature tensor. Specially, Einstein manifold with $k=0$ are called Ricci-flat manifolds.
\end{definition}
In general relativity, Einstein's equation with a cosmological constant $\Lambda$ is
$$
R_{ab}-\frac{1}{2}g_{ab}R+g_{ab}\Lambda=8\pi T_{ab}
$$
written in geometrized units with $G=c=1$. The stress-energy tensor $T_{ab}$ gives the matter and energy content of the underlying spacetime. In vacuum (a region of spacetime devoid of matter) $T_{ab}=0$, and Einstein's equation can be rewritten in the form (assuming that $n > 2$):
$$
R_{ab}=\frac{2\Lambda}{n-2}g_{ab}.
$$
Therefore, vacuum solutions of Einstein's equation are (Lorentzian) Einstein manifolds with $k$ proportional to the cosmological constant.\par
Other examples of Einstein manifolds include: any manifold with constant sectional curvature is an Einstein manifold, the complex projective space $\mathbb{CP}^n$ , with the Fubini$-$Study metric, Calabi$-$Yau manifolds admit an Einstein metric that is also K$\ddot{\rm{a}}$hler, with Einstein constant $k=0$.
\subsection{M-eigenvalue of Riemann curvature tensor is conformal invariant}
\begin{definition}
If $\Phi$ is a diffeomorphism for a Riemannian manifold $(M,g)$ to itself. If there is a positive smooth function $\lambda \in C^{\infty}(M)$ such that
$$
\tilde{g}=\Phi^{\star}g=\lambda^2 g,
$$ where $\Phi^{\star}g$ is the pulled back metric, then we call $\Phi$ is a conformal map on the Riemannian manifold $(M,g)$
\end{definition}
We find the Riemann curvature tensor and its M-eigenvalues are conformal invariant, that is the following theorem.
\begin{theorem}
Suppose $(M,g)$ is a m-dimensional Einstein manifold which is conformal flat, then the Riemann curvature tensor is conformal invariant up to a conformal factor, that is,
$$
\tilde{R}_{ijkl}=\lambda^2 R_{ijkl}.
$$
\end{theorem}
\begin{proof}
Since $(M,g)$ is an Einstein manifold, we have
$$
R_{ab}=k g_{ab}
$$
By definition, the pulled back metric tensor satisfies
$$
\tilde g_{ab}=\lambda^2 g_{ab}
$$
so we have
$$
\tilde g^{ab}=\lambda^{-2} g^{ab}
$$
so the Riemann curvature tensor
$$
\begin{aligned}
R^{l}_{\;kij}&=\frac{1}{m-2}(\delta^{l}_{\;i}R_{kj}-\delta^{l}_{\;j}R_{ki}
+g_{kj}g^{lp}R_{pi}-g_{ki}g^{lp}R_{pj})\\
&+\frac{R}{(m-1)(m-2)}(\delta^{l}_{\;j}g_{ki}-\delta^{l}_{\;i}g_{kj})\\
&=\frac{1}{m-2}(\delta^{l}_{\;i}k g_{kj}-\delta^{l}_{\;j}k g_{ki}
+g_{kj}g^{lp}k g_{pi}-g_{ki}g^{lp}k g_{pj})\\
&+\frac{R}{(m-1)(m-2)}(\delta^{l}_{\;j}g_{ki}-\delta^{l}_{\;i}g_{kj})
\end{aligned}
$$
so we have
$$
\begin{aligned}
\tilde{R}^{l}_{\;kij}&=\frac{1}{m-2}(\delta^{l}_{\;i}k \tilde{g}_{kj}-\delta^{l}_{\;j}k \tilde{g}_{ki}+\tilde{g}_{kj}\tilde{g}^{lp}k \tilde{g}_{pi}-\tilde{g}_{ki}\tilde{g}^{lp}k \tilde{g}_{pj})\\
&+\frac{R}{(m-1)(m-2)}(\delta^{l}_{\;j}\tilde{g}_{ki}-\delta^{l}_{\;i}\tilde{g}_{kj})\\
&=\lambda^2 R^{l}_{\;kij}
\end{aligned}
$$
\end{proof}
From the above theorem, we have
\begin{theorem}
If $\Phi$ is a diffeomorphism for a Riemannian manifold $(M,g)$ to itself. If there is a positive smooth function $\lambda \in C^{\infty}(M)$ such that
$$
\tilde{g}=\Phi^{\star}g=\lambda^2 g,
$$ where $\Phi^{\star}g$ is the pulled back metric, and $\tilde{R}^{l}_{\;kij}$ is the pulled back Riemann curvature tensor. If $(\zeta,x,y)$ is the M-eigentriple of $R^{l}_{\;kij}$ then $(k^2\zeta,x,y)$ is the M-eigentriple of $\tilde{R}^{l}_{\;kij}$.
\end{theorem}
\begin{proof}
From the above theorem we have
$$\tilde{R}^{l}_{\;kij}=k^2 R^{l}_{\;kij} $$
So we have
$$
\tilde{R}^{l}_{\;kij}y^i x^j y^k=k^2 R^{l}_{\;kij}y^i x^j y^k=k^2 \zeta x^l
$$
that is the result.
\end{proof}
The above theorem is to say, the M-eigentriple and the Riemannian curvature tensor of a Einstein manifold is conformal invariant up to a factor $k^2$. And for a conformal flat manifold, we have the following theorem:
\begin{theorem}
If $\Phi$ is a diffeomorphism for a Riemannian manifold $(M,g)$ to itself. If there is a positive smooth function $\lambda \in C^{\infty}(M)$ such that
$$
\tilde{g}=\Phi^{\star}g=\lambda^2 g,
$$ where $\Phi^{\star}g$ is the pulled back metric, and $\tilde{R}^{l}_{\;kij}$ is the pulled back Riemann curvature tensor. Then
$$
\tilde{\zeta}=k^2\left(\frac{(m-1)(\lambda+\mu)+R}{(m-2)(m-1)}\right)
$$
is the M-eigenvalue of $\tilde{R}^{l}_{\;kij}$.
\end{theorem}
\subsection{Sectional curvature and canonical form of Riemann curvature tensor}
In Riemannian geometry, the sectional curvature is one of the ways to describe the curvature of Riemannian manifolds. The sectional curvature $K(\pi)$ depends on a two-dimensional plane $\pi$ in the tangent space at a point $p$ of the manifold.\par
Given a Riemannian manifold and two linearly independent tangent vectors at the same point, $u$ and $v$, we can define the sectional curvature
$$
K(\pi)=\frac{R(u,v,u,v)}{G(u,v,u,v)}
$$
where $\pi$ is the two dimensional plane spanned by tangent vector $u$ and $v$, and
$$
R(u,v,u,v)=\langle R(u,v)u,v\rangle
$$
$$
G(u,v,u,v)=\langle u,u\rangle \langle v,v\rangle- \langle u,v\rangle^2.
$$
Considering the M-eigenvalue of the Riemann curvature tensor, we have the following results.
\begin{theorem}
Suppose $(M,g)$ is a $m$ dimensional conformal flat Riemannian manifold, $R^{i}_{\;jkl}$ is the Riemann curvature tensor, and $R_{ab}$ is the Ricci curvature tensor at $p\in M$. If $R_{ab}$ has eigenpair $(\lambda,x)$, and $(\mu,y)$, where $x$ and $y$ are orthogonal to each other and normalized, then the sectional curvature correspond to the 2-dimensional subspace
$$
\pi=span\{u,v\}\subseteq T_p(M)
$$
equals to the M-eigenvalue
$$
K(\pi)=\frac{R(u,v,u,v)}{G(u,v,u,v)}=\frac{(m-1)(\lambda+\mu)+R}{(m-2)(m-1)}.
$$
\end{theorem}
\begin{proof}
Suppose $\lambda$ is the M-eigenvalue satisfies
$$
R^{i}_{\;jkl} y^j x^k y^l=\lambda x^i
$$
then we have
$$
\lambda=\langle R^{i}_{\;jkl} y^j x^k y^l,x^i\rangle=R(x,y,x,y).
$$
Since $x$ and $y$ are orthogonal to each other and normalized, so we have
$$
K(\pi)=\lambda=\frac{(m-1)(\lambda+\mu)+R}{(m-2)(m-1)}
$$
is the sectional curvature.
\end{proof}
Sectional curvature is of great importance in Riemannian geometry, according to the following Lemma.
\begin{lemma}
Suppose $(M,g)$ is a Riemannian manifold, then the Riemann curvature tensor at $p\in M$ can be determined by the sectional curvature of all the 2-dimensional tangent subspaces at $p$ uniquely.
\end{lemma}
\begin{proof}
Suppose there is another 4-linear function $\bar{R}(X,Y,Z,W)$ satisfies the properties $(i)$ to $(iv)$ of the curvature tensor $R(X,Y,Z,W)$, and that for any two linearly independent tangent vectors $X,Y$ at $p$,
$$
\frac{\bar{R}(X,Y,X,Y)}{G(X,Y,X,Y)}=\frac{R(X,Y,X,Y)}{G(X,Y,X,Y)}.
$$
We will show that for any $X,Y,Z,W\in T_{p}(M)$
$$
\bar{R}(X,Y,Z,W)=R(X,Y,Z,W).
$$
If we let
$$
S(X,Y,Z,W)=\bar{R}(X,Y,Z,W)-R(X,Y,Z,W)
$$
then $S$ is also a 4-linear function satisfying the properties $(i)$ to $(iv)$ of curvature tensor and we have
$$
S(X,Y,X,Y)=0.
$$
We obtain
$$
S(X+Z,Y,X+Z,Y)=0.
$$
Expanding and using the properties of the function $S$, we have
$$
S(X,Y,Z,Y)=0,
$$
where $X,Y,Z$ are any three elements of $T_p(M)$. Thus
$$
S(X,Y+W,Z,Y+W)=0,
$$
and by expanding we obtain
$$
S(X,Y,Z,W)+S(X,W,Z,Y)=0.
$$
So we have
$$
\begin{aligned}
S(X,Y,Z,W)&=-S(X,W,Z,Y)=S(X,W,Y,Z)\\
&=-S(X,Z,Y,W)=S(X,Z,W,Y).
\end{aligned}
$$
On the other hand we have
$$
S(X,Y,Z,W)+S(X,Z,W,Y)+S(X,W,Y,Z)=0.
$$
Thus
$$
3S(X,Y,Z,W)=0.
$$
\end{proof}
As we know the eigenvalues and eigenvectors is very important to a matrix, since the matrix can be determined by them uniquely if they are linear independent, so we can introduce the canonical form of a matrix.\par
There are many kinds of tensor eigenvalues \cite{Qi,Qi3}, such as H-eigenpair $(\lambda,x)$ satisfies:
$$
\mathscr{A}x^{m-1}=\lambda x^{[m-1]},
$$
or Z-eigenpair $(\mu,y)$ satisfies:
$$
\mathscr{A}x^{m-1}=\lambda x.
$$
However, there is no theorem tells us that a tensor can be determined uniquely from those kinds of eigenpairs. So those kinds of tensor eigenpair cannot give us a classification of tensors. But from the above proof of lemma, we find the M-eigentriple can do those kind of things.
\begin{theorem}
If a 4th order tensor $A_{ijkl}$ satisfies the properties $(i)$ to $(iv)$, then the tensor can be determined by the linear independent M-eigentriples $(\lambda_i,x_i,y_i)$ uniquely.
\end{theorem}
\begin{proof}
Since the above theorem only use the properties $(i)$ to $(iv)$ of curvature tensor, we have the same result for normal tensors, not only for Riemann curvature tensor.
\end{proof}
This theorem tell us the M-eigentriple can be seen as the canonical form of a 4th order tensor.\par
Now we have found the relationship between the M-eigenvalue and sectional curvature, the following theorem which tells us the M-eigenvalue is of great importance to the conformal flat manifold.
\begin{theorem}
Suppose $(M,g)$ is a m dimension conformal flat Riemannian manifold, $p\in M$ is a point on the manifold, $R_{ab}$ is the Ricci curvature tensor, which is a $m\times m$ matrix. If $R_{ab}$ has m linear independent eigenpairs $(\lambda_1,x_1)$,$(\lambda_2,x_2)$, $\cdots$, $(\lambda_m,x_m)$, then the Riemann curvature tensor at $p$ can be determined by the $C^2_m$ M-eigenvalues.
\end{theorem}
\begin{proof}
Since $(M,g)$ is a m dimensional manifold, then the tangent space at a point
$p\in M$ is also m dimensional. So the number of different 2-dimensional subspace of the tangent space at $p$ is $C^2_m$. We can know from the above theorems that the $C^2_m$ M-eigenvalues are
$$
\lambda=\frac{(m-1)(\lambda_i+\lambda_j)+R}{(m-2)(m-1)}\ \ i,j=1,2,\cdots,m.
$$
From the lemma, the Riemann curvature tensor can be determined by these M-eigenvalues.
\end{proof}
\section{Complex M-eigenvalue and K$\ddot{\rm{\textbf{a}}}$hler manifold}
\subsection{Preliminaries}
In the research of modern mathematics, complex manifolds play a more and more important role, especially the K$\ddot{\rm{a}}$hler manifolds. K$\ddot{\rm{a}}$hler manifold is a complex manifold with a Riemannian metric which is invariant under the action of the complex structure $J$. At the same time the complex structure satisfies:
$$
\nabla J \equiv0,
$$
where $\nabla$ is the Riemannian connection. Therefore, the K$\ddot{\rm{a}}$hler manifold is a special kind of the Riemannian manifold with more interesting structures and characteristics.\par
In this section, we introduce a new kind of the complex M-eigenvalue and complex M-eigenvector based on complex manifolds and holomorphic sectional curvature.
\begin{definition}
Suppose $W$ is a real vector space. If there is a linear transformation:
$$
J : W\rightarrow W
$$
satisfies
$$
J^2=-id,
$$
then $J$ is called a complex structure on $W$.
\end{definition}
Suppose $V$ is an n dimensional complex vector space and it has a basis $\{e_1, e_2, \cdots, e_n\}$, then when the field is restricted to the real number, the vector space $V_{\mathbb{R}}$ is a 2n dimensional vector space which has basis $$\{e_1, e_2, \cdots, e_n, Je_1,Je_2,\cdots,Je_n\},$$ where $J$ is the complex structure of $V$.\par
Denote $e_{\bar{i}}=J e_{i}$, where $\bar{i}=n+i$, $1\leq i \leq n$. We use $\{\theta^{i}, \theta^{\bar{i}}\}$ to denote the dual basis of $\{e_{i}, e_{\bar{i}}\}$ of $V_{\mathbb{R}}^{\star}$, that is
$$
\theta^{\alpha}(e_{\beta})=\delta^{\alpha}_{\;\beta},\ 1\leq \alpha, \beta\leq 2n.
$$
and we also have
$$
J\theta_{i}=-\theta^{\bar{i}},\ J\theta^{\bar i}=\theta^{i}.
$$
\subsection{Hermite inner product}
For complex manifolds, we can define Hermite inner products on them.
\begin{definition}
Suppose $V$ is an $n$ dimensional complex vector space, $h:V\times V\rightarrow \mathbb{C}$ is a complex valued function defined on $V$, we call $h$ is an Hermite inner product if $h$ satisfies:\newline
\rm{(i)} $h$ is real-bilinear, that is $\forall x,y,z \in V, \lambda\in \mathbb{R}$
$$
h(x+\lambda z, y)=h(x,y)+\lambda h(z,y),
$$
$$
h(x,y+\lambda z)=h(x,y)+\lambda h(x,z).
$$
\rm{(ii)} $h$ is complex linear for the first variant, that is $\forall x,y \in V, \lambda\in \mathbb{C}$
$$
h(\lambda x,y)=\lambda h(x,y).
$$
\rm{(iii)} $\forall x,y \in V$
$$
h(y,x)=\overline{h(x,y)}.
$$
\rm{(iv)} $\forall x \in V$,
$$
h(x,x)\geq 0,
$$
$h(x,x)=0$ if and only if $x=0$.
\end{definition}\par
Decompose $h$ into real part and imaginary part:
$$
h(x,y)=g(x,y)+\sqrt{-1}k(x,y),
$$
where $g$ and $k$ are real valued real-bilinear functions on $V$.
Since
$$
g(Jx,y)+\sqrt{-1}k(Jx,y)=h(Jx,y)=h(\sqrt{-1}x,y)=\sqrt{-1}h(x,y),
$$
and
$$
\sqrt{-1}h(x,y)=-k(x,y)+\sqrt{-1}g(x,y),
$$
we have
$$
g(Jx,y)=-k(x,y),\ g(x,y)=k(Jx,y), \  \forall x,y \in V_{\mathbb{R}}.
$$
On the other hand,
$$
g(y,x)+\sqrt{-1}k(y,x)=h(y,x)=\overline{h(x,y)}=g(x,y)-\sqrt{-1}k(x,y),
$$
we get
$$
g(y,x)=g(x,y), \ k(y,x)=-k(x,y).
$$
Then we have
$$
g(Jx,Jy)=-k(x,Jy)=k(Jy,x)=g(y,x)=g(x,y),
$$
which means $g$ is invariant under the action of complex structure $J$ on $V_{\mathbb{R}}$.
\subsection{Coordinate expression of Hermite inner product}
Suppose $V$ is a complex vector field with basis $\{e_i\}, 1\leq i \leq n$ and $V^{\star}$ has dual basis $\{\omega^{i}\}, 1\leq i \leq n$, that is $\omega^{i}(e_j)=\delta^{i}_{\ j}$. The coordinate expression of $h$ is
$$
h_{ij}=h(e_{i},e_{j}), \  1\leq i,j\leq n.
$$
From above, we have $h_{ij}=\overline{h_{ji}}$.\par
Since $g$ is real-bilinear function on $V_{\mathbb{R}}$,
$$
g_{\alpha\beta}=g(e_{\alpha},e_{\beta}), \ 1\leq \alpha,\beta\leq 2n.
$$
Then we have
$$
g_{ij}=g_{\bar{i}\bar{j}}=g_{ji}=g_{\bar{j}\bar{i}},
$$
$$
g_{i\bar{j}}=-g_{\bar{i}j}=g_{\bar{j}i}=-g_{j\bar{i}}.
$$
So $h_{ij}=g_{ij}+\sqrt{-1}g_{i\bar{j}}$, where $\bar{i}=i+n$.
\subsection{Curvature tensor on K$\ddot{\rm{\textbf{a}}}$hler Manifolds}
\begin{lemma}
Suppose $(M,h)$ is a K$\ddot{a}$hler manifold, $J$ is the complex structure on $M$, $g=\Re(h)$. Then $(M,g)$ can be treated as a Riemannian manifold, which satisfies not only $(i)$ to $(iv)$ and also the following: $\forall X,Y \in T_{p}(M)$
$$
(1)~R(X,Y)\circ J=J\circ R(X,Y),
$$
$$
(2)~R(JX,JY)=R(X,Y).
$$
\end{lemma}
The coordinate expression of Riemman curvature tensor is
$$
R_{\beta \alpha \gamma \delta}=g\left(R\left(\frac{\partial}{\partial x^{\gamma}},\frac{\partial}{\partial x^{\delta}}\right)\frac{\partial}{\partial x^{\alpha}},\frac{\partial}{\partial x^{\beta}}\right),\ 1\leq \alpha,\beta,\gamma,\delta\leq 2n.
$$
As a complex manifold, we define the complex curvature tensor:
$$
K_{ijkl}=2h\left(R\left(\frac{\partial}{\partial z^{k}},\frac{\partial}{\partial \overline{z^{l}}}\right)\frac{\partial}{\partial z^{i}},\frac{\partial}{\partial z^{j}}\right), \ 1\leq i,j,k,l\leq n,
$$
where
$$
\frac{\partial}{\partial z^i}=\frac{1}{2}\left(\frac{\partial}{\partial x^i}-\sqrt{-1}\frac{\partial}{\partial y^i}\right),\ 1\leq i\leq n,
$$
$$
\frac{\partial}{\partial \overline{z^i}}=\frac{1}{2}\left(\frac{\partial}{\partial x^i}+\sqrt{-1}\frac{\partial}{\partial y^i}\right),\ 1\leq i\leq n.
$$
\begin{lemma}
The relationship between the coordinate expression of complex curvature tensor and Riemann curvature tensor is:
$$
K_{ijkl}=(R_{ijkl}-R_{i\bar{j}k\bar{l}})+\sqrt{-1}(R_{i\bar{j}kl}+R_{ijk\bar{l}}),\ 1\leq i,j,k,l\leq n.
$$
\end{lemma}
Now we give the definition of complex M-eigenvalue.
\begin{definition}
Suppose $(M,h)$ is an $n$ dimensional K$\ddot{a}$hler manifold, $K_{ijkl}$ is the complex curvature tensor on M, we call $\sigma \in \mathbb{C}$ and $z \in \mathbb{C}^n$ is the complex M-eigenvalue and eigenvector if
$$
\begin{cases}
&K_{ijkl}\overline{z^j}z^k \overline{z^l}=\lambda \overline{z_i},\ 1\leq i,j,k,l\leq n.\\
&h(z^i,z^j)=h_{ij}z^i \overline{z^j}=1, \ 1\leq i,j\leq n.\\
\end{cases}
$$
\end{definition}
Given an $n$ dimensional K$\ddot{\rm{a}}$hler manifold $(M,h)$, for each point $p\in M$, $X\in T_{p}M$, the section curvature $K(X)\equiv K(X,JX)$ of the 2 dimensional section $[X\wedge JX]$ is called the holomorphic curvature tensor at $p$ of the direction $X$, where
$$
K(X)=-\frac{R(X,JX,X,JX)}{g(X,X)g(JX,JX)-(g(X,JX))^2}.
$$
Since $JX\perp X$, $g(X,X)g(JX,JX)-(g(X,JX))^2=g(X,X)^2$, so
$$
K(X)=-\frac{R(X,JX,X,JX)}{g(X,X)^2}.
$$
Now we give the coordinate expression of holomorphic curvature tensor.
\begin{lemma}
Suppose $(M,h)$ is an $n$ dimensional K$\ddot{a}$hler manifold, $p\in M$, $(U,z^i)$ is the complex local coordinate at $p$, $z^i=x^i+\sqrt{-1}y^i$,  $X=X^i \frac{\partial}{\partial x^i}+X^{\bar{i}} \frac{\partial}{\partial y^i} \in T_{p}M$, then the coordinate expression of holomorphic curvature tensor is
$$
K(X)=\frac{K_{ijkl}Z^i\overline{Z^j}Z^k\overline{Z^l}}{(h_{ij}Z^i\overline{Z^j})^2},
$$
where $Z^i=X^i+\sqrt{-1}X^{\bar{i}}$.
\end{lemma}
\begin{lemma}
Suppose $(M,h)$ is K$\ddot{a}$hler manifold with constant holomorphic sectional curvature $c$, then
$$
K^{i}_{\;jkl}=\frac{c}{2}\left(\delta^{j}_{\;i}h_{kl}+\delta^{j}_{\;k}h_{il}\right),
$$
where $h_{ij}=h\left(\frac{\partial}{\partial z^j},\frac{\partial}{\partial z^i}\right)$.
\end{lemma}
\begin{theorem}
Suppose $(M,h)$ is K$\ddot{a}$hler manifold with constant holomorphic sectional curvature $c$, then its complex M-eigenvalue of the complex curvature tensor $K_{ijkl}$ is
$$
\sigma =c.
$$
and all the vectors in $\bold{CP}^{n-1}$ are its complex eigenvectors.
\end{theorem}
\begin{proof}
The complex eigenvalue and eigenvector satisfies
$$
\begin{cases}
&K_{ijkl}\overline{z^j}z^k \overline{z^l}=\sigma \overline{z_i},\\
&h_{ij}z^i \overline{z^j}=1,\\
\end{cases}
$$
while
$$
K_{ijkl}=h_{ij}h_{kl}+h_{il}h_{kj},
$$
So
$$
\begin{aligned}
K_{ijkl}\overline{z^j}z^k \overline{z^l}&=\frac{c}{2}(h_{ij}h_{kl}+h_{il}h_{kj})\overline{z^j}z^k \overline{z^l}\\
&=\frac{c}{2}(h_{ij}\overline{z^j}+h_{il}\overline{z^l})\\
&=c \cdot\overline{z^i}
\end{aligned}
$$
so the complex eigenvalue is the constant holomorphic sectional curvature $c$.
\end{proof}

\section{Example: de Sitter Spacetime}
\subsection{Introduction to de Sitter Spacetime}
In mathematics and physics, a de Sitter space is the analog in Minkowski space, or spacetime, of a sphere in ordinary Euclidean space. The $n$-dimensional de Sitter space, denoted $dS_n$, is the Lorentzian manifold analog of an $n$-sphere (with its canonical Riemannian metric); it is maximally symmetric, has constant positive curvature, and is simply connected for $n$ at least $3$. De Sitter space and anti-de Sitter space are named after Willem de Sitter (1872$-$1934), professor of astronomy at Leiden University and director of the Leiden Observatory. Willem de Sitter and Albert Einstein worked in the 1920s in Leiden closely together on the spacetime structure of our universe.\par
In the language of general relativity, de Sitter space is the maximally symmetric vacuum solution of Einstein's field equations with a positive cosmological constant $\Lambda$  (corresponding to a positive vacuum energy density and negative pressure). When $n = 4$ ($3$ space dimensions plus time), it is a cosmological model for the physical universe.\par
De Sitter space was also discovered, independently, and about the same time, by Tullio Levi-Civita.\par
More recently it has been considered as the setting for special relativity rather than using Minkowski space, since a group contraction reduces the isometry group of de Sitter space to the Poincar\'{e} group, allowing a unification of the spacetime translation subgroup and Lorentz transformation subgroup of the Poincar\'{e} group into a simple group rather than a semi-simple group. This alternative formulation of special relativity is called de Sitter relativity.

\subsection{Curvature tensors of de Sitter Spacetime}
The line element of de Sitter Spacetime can be written as:
$$
ds^2=-\left(1-\frac{r^2}{a^2}\right)\;dt^2+\frac{1}{1-\frac{r^2}{a^2}}\;dr^2+r^2 \;d\theta^2+r^2\sin^2\theta \;d \phi^2.
$$
So the metric tensor can be written as:
$$
g_{ab}=
\begin{bmatrix}
-\left(1-\frac{r^2}{a^2}\right)&0&0&0\\
0&\frac{1}{1-\frac{r^2}{a^2}}&0&0\\
0&0&r^2&0\\
0&0&0&r^2\sin^2\theta\\
\end{bmatrix}
$$
The Riemann curvature tensor is a $4\times 4\times 4\times 4$ tensor, which has $24$ non-zero elements:
$$
R^{1}_{\;1\cdot\cdot}=
\begin{bmatrix}
0&0&0&0\\
0&0&0&0\\
0&0&0&0\\
0&0&0&0\\
\end{bmatrix},
R^{1}_{\;2\cdot\cdot}=
\begin{bmatrix}
0&\frac{1}{a^2-r^2}&0&0\\
-\frac{1}{a^2-r^2}&0&0&0\\
0&0&0&0\\
0&0&0&0\\
\end{bmatrix},
$$
$$
R^{1}_{\;3\cdot\cdot}=
\begin{bmatrix}
0&0&\frac{r^2}{a^2}&0\\
0&0&0&0\\
-\frac{r^2}{a^2}&0&0&0\\
0&0&0&0\\
\end{bmatrix},
R^{1}_{\;4\cdot\cdot}=
\begin{bmatrix}
0&0&0&\frac{r^2-\sin^2\theta}{a^2}\\
0&0&0&0\\
0&0&0&0\\
-\frac{r^2-\sin^2\theta}{a^2}&0&0&0\\
\end{bmatrix},
$$
$$
R^{2}_{\;1\cdot\cdot}=
\begin{bmatrix}
0&\frac{a^2-r^2}{a^4}&0&0\\
-\frac{a^2-r^2}{a^4}&0&0&0\\
0&0&0&0\\
0&0&0&0\\
\end{bmatrix},
R^{2}_{\;2\cdot\cdot}=
\begin{bmatrix}
0&0&0&0\\
0&0&0&0\\
0&0&0&0\\
0&0&0&0\\
\end{bmatrix},
$$
$$
R^{2}_{\;3\cdot\cdot}=
\begin{bmatrix}
0&0&0&0\\
0&0&\frac{r^2}{a^2}&0\\
0&-\frac{r^2}{a^2}&0&0\\
0&0&0&0\\
\end{bmatrix},
R^{2}_{\;4\cdot\cdot}=
\begin{bmatrix}
0&0&0&0\\
0&0&0&\frac{r^2\sin^2\theta}{a^2}\\
0&0&0&0\\
0&-\frac{r^2\sin^2\theta}{a^2}&0&0\\
\end{bmatrix},
$$

$$
R^{3}_{\;1\cdot\cdot}=
\begin{bmatrix}
0&0&\frac{a^2-r^2}{a^4}&0\\
0&0&0&0\\
-\frac{a^2-r^2}{a^4}&0&0&0\\
0&0&0&0\\
\end{bmatrix},
R^{3}_{\;2\cdot\cdot}=
\begin{bmatrix}
0&0&0&0\\
0&0&\frac{1}{-a^2+r^2}&0\\
0&\frac{1}{a^2-r^2}&0&0\\
0&0&0&0\\
\end{bmatrix},
$$

$$
R^{3}_{\;3\cdot\cdot}=
\begin{bmatrix}
0&0&0&0\\
0&0&0&0\\
0&0&0&0\\
0&0&0&0\\
\end{bmatrix},
R^{3}_{\;4\cdot\cdot}=
\begin{bmatrix}
0&0&0&0\\
0&0&0&0\\
0&0&0&\frac{r^2\sin^2\theta}{a^2}\\
0&0&-\frac{r^2\sin^2\theta}{a^2}&0\\
\end{bmatrix},
$$

$$
R^{4}_{\;1\cdot\cdot}=
\begin{bmatrix}
0&0&0&\frac{a^2-r^2}{a^4}\\
0&0&0&0\\
0&0&0&0\\
-\frac{a^2-r^2}{a^4}&0&0&0\\
\end{bmatrix},
R^{4}_{\;2\cdot\cdot}=
\begin{bmatrix}
0&0&0&0\\
0&0&0&\frac{1}{-a^2+r^2}\\
0&0&0&0\\
0&-\frac{1}{-a^2+r^2}&0&0\\
\end{bmatrix},
$$

$$
R^{4}_{\;3\cdot\cdot}=
\begin{bmatrix}
0&0&0&0\\
0&0&0&0\\
0&0&0&-\frac{r^2}{a^2}\\
0&0&\frac{r^2}{a^2}&0\\
\end{bmatrix},
R^{4}_{\;4\cdot\cdot}=
\begin{bmatrix}
0&0&0&0\\
0&0&0&0\\
0&0&0&0\\
0&0&0&0\\
\end{bmatrix}.
$$

The Ricci curvature tensor is:
$$
\begin{bmatrix}
\frac{3(-a^2+r^2)}{}&0&0&0\\
0&\frac{3}{a^2-r^2}&0&0\\
0&0&\frac{3r^2}{a^2}&0\\
0&0&0&\frac{3r^2\sin^2\theta}{a^2}\\
\end{bmatrix}.
$$
and the scalar curvature is:
$$
R=\frac{12}{a^2}.
$$
So the eigenvalue and eigenvector of Ricci curvature tensor is:
$$
\lambda_1=\frac{3(-a^2+r^2)}{a^4},\ x_1=(1,0,0,0),
$$
$$
\lambda_2=\frac{3}{a^2-r^2},\ x_2=(0,1,0,0),
$$
$$
\lambda_3=\frac{3r^2}{a^2},\ x_3=(0,0,1,0),
$$
$$
\lambda_4=\frac{3r^2\sin^2\theta}{a^2},\ x_4=(0,0,0,1).
$$
So according to our theorem, the M-eigentriple of de Sitter spacetime is:
$$
(\lambda,\zeta_i,\zeta_j)=\left(\frac{3a^2(\lambda_i+\lambda_j)+12}{6a^2},x_i,x_j\right),i,j=1,2,3,4\ \ (i\neq j).
$$
Since de Sitter spacetime is conformal flat, the eigenvalue and eigenvector of Ricci curvature tensor in the tangent space of a point on the manifold can determine the Riemann curvature tensor uniquely, which means the sectional curvature correspond to the 2-dimension subspace spanned by $x_i$ and $x_j$ can determine the Riemann curvature tensor. In some sense the Ricci curvature tensor can be seen as the canonical form of Riemann curvature tensor.
\section{Conclusion}
Starting from the M-eigenvalue of two dimensional and three dimensional tensor, we generalize the M-eigenvalue theory to $m$ dimensional conformal flat manifolds and K$\ddot{\rm{a}}$hler manifolds, which can be seen as a generalization of Xiang, Qi and Wei's work on Riemann curvature tensor. However, we can just obtain a few of the M-eigenvalues and cannot calculate all the M-eigenvalues, and we even do not know the number of M-eigenvalues. In the next step of our research, we will consider how the M-eigenvalues evolution when the elements of Riemann curvature tensor is time-varying. Also we will consider the relationship between geodesic deviation and M-eigenvalues which may shed light on the physical meaning of the M-eigenvalues of Riemann curvature tensor.

\end{document}